\documentclass[12pt]{article}
\usepackage{amsmath,amssymb,amstext,dsfont,fancyvrb,float,fontenc,graphicx,subfigure,theorem,hyperref,mathtools}

\usepackage{tikz,everypage,bm}
\usepackage[letterpaper]{geometry}
\ifx\volno\undefined\def\volno{0}\fi
\ifx\volyear\undefined\def\volyear{2017}\fi
\ifx\pagno\undefined\def\pagno{000--000}\fi

\newfont{\footsc}{cmcsc10 at 8truept}
\newfont{\footbf}{cmbx10 at 8truept}
\newfont{\footrm}{cmr10 at 10truept}

\usepackage{fancyhdr}
\usepackage{relsize}
\usepackage{sectsty}
\allsectionsfont{\larger[-1]} 

\renewcommand\paragraph{\@startsection{paragraph}{4}{\z@}
                                    {2ex \@plus.5ex \@minus.2ex}
                                    {-1em}
                                    {\normalfont\normalsize\bfseries}}

\renewcommand\subparagraph{\@startsection{subparagraph}{5}{\parindent}
                                       {2ex \@plus.5ex \@minus .2ex}
                                       {-1em}
                                      {\normalfont\normalsize\bfseries}}

\newlength{\BiblioSpacing}
\renewenvironment{thebibliography}[1]{
\begin{oldthebibliography}{#1}
\setlength{\parskip}{\BiblioSpacing}
\setlength{\itemsep}{\BiblioSpacing}
}
{
\end{oldthebibliography}
}

\usepackage[strict]{changepage}
\def\abstractname{Abstract -}   % <-----------------
\def\abstract{\begin{adjustwidth}{1cm}{1cm} \par    \footnotesize \noindent {\bf \abstractname} 
\def\endabstract{ \end{adjustwidth} \smallskip }}
{\theorembodyfont{\itshape}\newtheorem{theorem}{Theorem}[section]}
{\theorembodyfont{\itshape}\newtheorem{proposition}[theorem]{Proposition}}
{\theorembodyfont{\itshape}\newtheorem{definition}[theorem]{Definition}}
{\theorembodyfont{\itshape}\newtheorem{lemma}[theorem]{Lemma}}
{\theorembodyfont{\itshape}\newtheorem{corollary}[theorem]{Corollary}}
{\theorembodyfont{\rm}}
{\theorembodyfont{\rm}}
{\theorembodyfont{\rm}}
{\theorembodyfont{\rm }\newtheorem{example}[theorem]{Example}}

\title{\Large\bf Sequences, $q$-multinomial Identities, Integer Partitions with Kinds, and Generalized Galois Numbers}
\author{\sc A. Avalos and M. Bly}
\begin{document}
\setcounter{page}{1}
%\date{}
\maketitle
\thispagestyle{fancy}

\vskip 1.5em

\begin{abstract}
Using sequences of finite length with positive integer elements and the inversion statistic on such sequences, a collection of binomial and multinomial identities are extended to their $q$-analog form via combinatorial proofs. Using the major index statistic on sequences, a connection between integer partitions with kinds and finite differences of the coefficients of generalized Galois numbers is established.
\end{abstract}
 
%malleable
 
\begin{MSC}
05A05; 05A15; 05A17; 05A19.
\end{MSC}

\section{Introduction} 
Broadly speaking, the primary topic of study within this paper is $q$-analogs. Introduced in the 19th century \cite{uber}, $q$-analogs are found within many areas of mathematics and related fields, including hypergeometric series \cite{hyper2,hyper3,hyper4,hyper1}, elliptic integrals \cite{el1,el2}, complex nonlinear dynamics \cite{com1,com2}, quantum calculus \cite{quant1,quant2}, and string theory \cite{str1,str2}. Given an expression, a $q$-analog is simply a corresponding expression parameterized by $q$ such that the limit as $q$ approaches 1 yields the original expression. A specific example that is central to this paper is the $q$-multinomial coefficient, which is a particular $q$-analog of multinomial coefficients. Stated explicitly in binomial form \[ \lim_{q \rightarrow 1} {n \choose k}_q \, = \, {n \choose k} \, . \]
These expressions can occur quite naturally: when considering the $k$-dimensional subspaces of an $n$-dimensional vector space over a finite field with $q$ elements, the number of such subspaces is the $q$-binomial coefficient ${n \choose k}_q\,$.

%The central element in this paper is the $q$-multinomial coefficient. Introduced in the 19th century [ref], this object of study is widely found within many areas of mathematics and related fields, including hypergeometric series [refs], elliptic integrals [refs], complex nonlinear dynamics [refs], quantum calculus [refs], and string theory [refs]. Simply stated, $q$-multinomial coefficients are algebraic expressions in a parameter $q$ such that the limit as $q$ approaches 1 yields the classical multinomial coefficient. More explicitly, \[ \lim_{q \rightarrow 1} {n \choose k}_q \, = \, {n \choose k} \, . \]
%These expressions in terms of $q$ can occur quite naturally: when considering the number of $k$-dimensional subspaces of an $n$-dimensional vector space over a finite field with $q$ elements, the number of such subspaces is precisely ${n \choose k}_q\,$. 

Our approach utilizes sequences of finite length with positive integer entries to provide an accessible treatment of $q$-multinomial coefficients. We will focus on the enumerative combinatorics of such sequences using a pair of elementary discrete statistics: the inversion statistic in Sections 1-2 and the major index statistic in Section 3. This viewpoint yields an intuitive understanding of $q$-multinomial coefficients to facilitate smooth navigation of the results of this paper. 

Two additional key elements in this paper are Galois numbers and integer partitions. The Galois numbers, originally introduced by Goldman and Rota \cite{rota}, are obtained from the $q$-multinomial coefficients. Imagine Pascal's $q$-triangle (see Figure \ref{fig:triangle}), a triangular array containing the expressions ${n \choose k}_q$ rather than the conventional numbers ${n \choose k}$. The $n^{\text{th}}$ Galois number is simply the sum of the expressions in the $n^{\text{th}}$ row of Pascal's $q$-Triangle. On the other hand, an integer partition is a finite series of nonincreasing positive integers. They are pervasive in combinatorics \cite{comb1,comb2,comb3,comb4} and also arise in the study of: number theory \cite{num1,num2,num3}, symmetric polynomials \cite{sym2,sym1}, and group representation theory \cite{rep1,rep2}.

%Throughout the paper, we will refer to the set $\{ \, 1, \, 2, \, \ldots, \, m \, \}$ as $[m]$ and the set sequences of length $n$ whose elements include $k_1$ 1's, $\ldots$, $k_m$ m's from the set $[m]$ as $\mathcal{S}_n^{\:\! m}(k_1, \ldots, k_m)\,$. In Section 1, we will establish essential definitions and ideas. In Section 2, we will concisely develop a robust collection of some classical and other less so classical $q$-multinomial identities. In Section 3, we will demonstrate a connection between integer partitions with kinds and finite differences of the coefficients of generalized Galois numbers.

\begin{figure}[!htb]
\begin{center}
$\begin{array}{ccccccc}
 & & & {0 \choose 0}_q & & & \\[3pt]
 & & {1 \choose 0}_q & & {1 \choose 1}_q & & \\[3pt]
 & {2 \choose 0}_q & & {2 \choose 1}_q & & {2 \choose 2}_q & \\[3pt]
 {3 \choose 0}_q & & {3 \choose 1}_q & & {3 \choose 2}_q
 & & {3 \choose 3}_q 
\end{array}$
\end{center}
\caption{\label{fig:triangle} Rows $0$ through $3$ of Pascal's $q$-Triangle}
\end{figure}

Throughout this paper, we will refer to the set $\{ \, 1, \, 2, \, \ldots, \, m \, \}$ as $[m]$ and the set sequences of length $n$ whose elements include $k_1$ 1's, $\ldots$, $k_m$ m's from the set $[m]$ as $\mathcal{S}_n^{\:\! m}(k_1, \ldots, k_m)\,$. In Section 1, we will establish some essential definitions and ideas. In Section 2, we will use the inversion statistic to concisely develop a collection of classical and other less so classical $q$-multinomial identities. In Section 3, we will introduce the major index statistic and use it to demonstrate a connection between generalizations of integer partitions and Galois numbers.

%This paper is the result of an investigation into a number of $q$-binomial and $q$-multinomial identities. To begin, we will establish essential definitions and ideas. In Section 2, we will concisely develop a robust collection of some classical and other less so classical binomial/multinomial identities in their $q$-analog form. In Section 3, we will demonstrate a connection between finite differences of the coefficients of generalized Galois numbers and integer partitions with kinds.

%At the heart of this paper is a motivation to present proofs of results using combinatorial justification. Along the way, we will encounter a number of objects of regular study in discrete mathematics: the set $[m]\,$, namely the set $\{ \, 1, \, 2, \, \ldots, \, m \, \}\,$; the set $\mathcal{S}_n^{\:\! m}(k_1, \ldots, k_m)\,$, namely the set of sequences of length $n$ whose elements include $k_1$ 1's, $\ldots$, $k_m$ m's from the set $[m]\,$; the inversion and major index statistics on sequences; and partitions of a positive integer $n\,$, namely sums of nonincreasing positive integers that add to $n\,$.

\subsection{Inversion Statistic}

To begin, we will introduce the well-known inversion statistic and a couple observations relevant to results in this paper. A general treatment of the inversion statistic from the literature can be found in \cite{Enum}.

\begin{definition}\label{def:inversions}
Let $n,m$ be nonnegative integers, and let $\sigma = ( \, \sigma_1, \, \ldots, \, \sigma_n \,)$ be a sequence  whose elements are from the set $[m]\,$. Then,
\[ \emph{inv}(\sigma) \, \coloneqq \, \left| \, \{ \, (a,b) \, \mid \, a<b \;\, \emph{and} \;\, \sigma_a > \sigma_b \, \} \, \right| \, . \]
If a particular $\sigma_a$ is fixed, ordered pairs of the form $(a,b)$ that are accounted for by \emph{inv(}$\sigma$\emph{)} shall be referred to as the \textbf{inversions induced by $\bm{\sigma_a}$} or simply \emph{i}$\left(\sigma_a\right)$\,. Should a particular $\sigma_b$ be fixed, ordered pairs of the form $(a,b)$ that are accounted for by \emph{inv(}$\sigma$\emph{)} shall be referred to as the \textbf{inversions received by $\bm{\sigma_b}$} or simply \emph{r}$\left(\sigma_b\right)$\,.
\end{definition}

Figure \ref{fig:sequences} contains some examples.

\begin{figure}[!htb]
\begin{center}
$\begin{array}{ccccc}
2211 & \hspace{.25in} & 2121 & \hspace{.25in} & 2112 \\
\text{inv}(\sigma)=4 && \text{inv}(\sigma)=3 && \text{inv}(\sigma)=2 \\[16pt]
1221 && 1212 && 1122 \\
\text{inv}(\sigma)=2 && \text{inv}(\sigma)=1 && \text{inv}(\sigma)=0
\end{array}$
\end{center}
\caption{\label{fig:sequences} All sequences of length $4$ with two $2$s and two $1$s.}
\end{figure}

\begin{proposition}\label{prop:received-induced}
Let $n,m$ be nonnegative integers, and let $\sigma = ( \, \sigma_1, \, \ldots, \, \sigma_n \,)$ be a sequence  whose elements are from the set $[m]\,$. Then, \[ \emph{inv}(\sigma) \, = \, \sum_{a \in [n]} \emph{i}\left( \sigma_a \right) \, = \, \sum_{b \in [n]} \emph{r}\left( \sigma_b \right) \, . \]
\end{proposition}

\begin{proof}
Observe the unions expressed below are disjoint.
\[ \{ \, (a,b) \, \mid a < b \, \} \, = \, \bigcup_{a \in [n]} \{ \, (a,b) \, \mid \, a<b \, \} \, = \, \bigcup_{b \in [n]} \{ \, (a,b) \, \mid \, a<b \, \} \, . \]
The result follows from the above statement of equality, and the definitions of: inversions, induced inversions, and received inversions.\end{proof}

\begin{corollary}\label{cor:received-induced}
Let $n,m$ be nonnegative integers, and let $\sigma = ( \, \sigma_1, \, \ldots, \, \sigma_n \,)$ be a sequence whose elements are from the set $[m]\,$. Then, \[ \emph{inv}(\sigma) \, = \, \sum_{\sigma_a \:\! \geq \:\! 2} \emph{i}\left( \sigma_a \right) \, = \, \sum_{\sigma_b \:\! \leq \:\! m-1} \emph{r}\left( \sigma_b \right) \, . \]
\end{corollary}

\begin{proof}
When $\sigma_a$ is equal to $1$ the value of i$\left( \sigma_a \right)$ equals zero. Similarly, when $\sigma_b$ is equal to $m$ the value of r$\left( \sigma_b \right)$ equals zero.
\end{proof}

\subsection{q-binomial and q-multinomial Coefficients}

We will now introduce a formal definition of the $q$-binomial coefficient. The following definition, though somewhat unconventional, is inspired by \cite{Knuth}. It will lead to an intuitive interpretation in terms of the inversion statistic, allowing for a smooth transition to their multinomial counterparts.

\begin{definition}\label{def:q-binomial}
Let $n,k$ be nonnegative integers such that $n\geq k\,$, and let $q$ be an indeterminate. Then
\[ {n \choose k}_q \, \coloneqq \, \sum_{\substack{E \subseteq [n] \\ |E|=k}} {q{\Large \strut}}^{\Big(\, \underset{i=1}{\overset{k}{\scriptstyle \sum}} (n-e_i)-(k-i) \, \Big) } \, .  \] where $E \, = \, \{ \, e_1, \, \ldots, \, e_k \, \}$ with $e_i < e_{i+1}$ for every $1\leq i \leq k-1\,$.
\end{definition}

Noting that the number of subsets of $[n]$ of cardinality $k$ is exactly ${n \choose k}\,$, we can see that letting $q \rightarrow 1$ yields the corresponding standard binomial coefficient.

Figure \ref{fig:q-binomial} contains some examples. Observe the parallelism between Figures \ref{fig:sequences} and \ref{fig:q-binomial}.

\begin{figure}[!htb]
\begin{center}
$\begin{array}{ccccc}
\{\, 1,2\,\} & \hspace{.25in} & \{\, 1,3\,\} & \hspace{.25in} & \{\, 1,4\,\} \\[1.5pt]
q^4 && q^3 && q^2 \\[16pt]
\{\, 2,3\,\} && \{\, 2,4\,\} && \{\, 3,4\,\} \\
q^2 && q^1 && q^0
\end{array}$
\end{center}
\caption{\label{fig:q-binomial} The sets associated with the terms of $\,{4 \choose 2}_q\, = \, q^4+q^3+2q^2+q+1\,$.}
\end{figure}

\begin{proposition}\label{prop:inv_binom}
If $n,k$ are nonnegative integers such that $n\geq k\,$ and $q$ is an indeterminate, then
\[ {n \choose k}_q \, = \, \sum_{\sigma \in \mathcal{S}_n^{\:\!2}(k,n-k)} q^{\emph{inv}(\sigma)} \, . \]
\end{proposition}

\begin{proof}
Let $E \subseteq [n]$ be of cardinality $k\,$, and let $\sigma = ( \, \sigma_1, \, \ldots, \, \sigma_n \,)$ be the sequence in $\mathcal{S}_n^{\:\!2}(k,n-k)$ such that $\sigma_a$ is $2$ precisely when $a \in E\,$. Fix some $a \in E$ and consider $\sigma_a\,$. The ordered pairs $(a,b)$ accounted for by $\text{inv}(\sigma)$ correspond to indices $b$ such that $\sigma_b$ is $1\,$ and $b>a$. Notice that $n-e_i$ equals $n-a$ and counts the number of indices $j$ such that $j>a\,$. Also notice that $k-i$ counts the numbers of elements $\sigma_j$ such that $j>a$ and $\sigma_j$ is $2\,$. Hence, $(n-e_i)-(k-i)$ counts all ordered pairs $(a,b)$ of interest. The result follows from observing that every $\sigma \in \mathcal{S}_n^{\:\!2}(k,n-k)$ can be attained similarly by some $E \subseteq [n]\,$. \end{proof}
$\;$\\[-1.5em]

In other words, the polynomial ${n \choose k}_q$ is the generating function for the statistic of inversions on the set $\mathcal{S}_n^{\:\!2}(k,n-k)\,$, a standard result which can be found in \cite{Enum}. The following definition and proposition provides the well-known $q$-multinomial generalization.

\begin{definition}
If $m,n,k_1,\ldots,k_m$ are nonnegative integers such that $k_1+\cdots + k_m=n\,$, then 
\[{n \choose {k_1, \, \ldots, \, k_m}}_q \, \coloneqq \, {n \choose k_m}_q{{n-k_m} \choose k_{m-1}}_q \cdots {{n-k_m-\cdots -k_2} \choose k_1}_q \, .\]
\end{definition}

\begin{proposition}\label{prop:inv_multinom}
If $m,n,k_1, \ldots, k_m$ are nonnegative integers such that $k_1+\cdots + k_m=n\,$, then
\[ {n \choose {k_1, \, \ldots, \, k_m}}_q \, = \, \sum_{\sigma \in \mathcal{S}_n^{\:\! m}(k_1, \ldots, k_m)} q^{\emph{inv}(\sigma)} \, . \]
\end{proposition}

\begin{proof}
Fix a sequence $\sigma = ( \, \sigma_1, \, \ldots, \, \sigma_n \,)$ in $\mathcal{S}_n^{\:\!m}(k_1, \ldots, k_m)\,$. Note that the inversions induced by all $\sigma_a$ for which $\sigma_a$ equals $m$ correspond to ordered pairs $(a,b)$ such that $\sigma_b$ is less than $m\,$. By Proposition \ref{prop:inv_binom}, it follows that ${n \choose k_m}_q$ corresponds precisely to inversions induced by all $\sigma_a$ equal to $m\,$.

Further observe that inversions induced by all $\sigma_a$ for which $\sigma_a$ equals $m-1$ correspond to ordered pairs $(a,b)$ such that $\sigma_b$ is less than $m-1\,$. In particular, no such $(a,b)$ will correspond to a $\sigma_b$ equal to $m\,$. As such, Proposition \ref{prop:inv_binom} applies to the subsequence of $\sigma$ containing the $n-k_m$ elements of $\sigma$ that do not equal $m\,$, and it follows that ${{n-k_m} \choose k_{m-1}}_q$ corresponds precisely to inversions induced by all $\sigma_a$ equal to $m-1\,$.

A similar argument holds for the remaining elements of $\sigma\,$.
\end{proof}

\subsection{Fundamental Sequences}
We will introduce an additional definition, inspired by \cite{Wilson} and named after P.A. MacMahon \cite{Mac}, that will be especially helpful in establishing the results of Section 3. 

\begin{definition}\label{def:Fund_Set}
Let $m,n$ be nonnegative integers. Define the \textbf{fundamental Mahonian set $\bm{\mathcal{F}_n^{\:\!m}}$} to be\[ \mathcal{F}_n^{\:\!m} \, \coloneqq \, \left\lbrace \, (F_1,\ldots,F_m) \,\; \big\vert \;\, |F_1|+\cdots+|F_m|=n \;\; \emph{and} \;\; \forall \, x \in F_{j+1}\, , \; x \in [K_j] \, \right\rbrace \, , \]where each $F_i$ is a (possibly empty) multiset of nonnegative integers and $K_j$ is equal to $|F_1|+\cdots+|F_j|\,$. Elements within $\mathcal{F}_n^{\:\!m}$ shall be referred to as \textbf{fundamental sequences}.
\end{definition}

Figure \ref{fig:fund_seq} contains some examples. Observe the parallelism between Figures \ref{fig:sequences} and \ref{fig:fund_seq}.

\begin{figure}[!htb]
\begin{center}
$\begin{array}{ccccc}
\big( \, \{0,0\} \, , \, \{2,2\} \, \big) && \big( \, \{0,0\} \, , \, \{2,1\} \, \big) && \big( \, \{0,0\} \, , \, \{2,0\} \, \big) \\[16pt]
\big( \, \{0,0\} \, , \, \{1,1\} \, \big) && \big( \, \{0,0\} \, , \, \{1,0\} \, \big) && \big( \, \{0,0\} \, , \, \{0,0\} \, \big)
\end{array}$
\end{center}
\caption{\label{fig:fund_seq} Fundamental sequences in $\mathcal{F}_4^{\:\! 2}$ for which $|F_2|=2$ and $|F_1|=2\,$.}
\end{figure}

\begin{proposition}\label{prop:fundset_bij}
If $m,n$ are nonnegative integers, then \[ \left| \, \mathcal{S}_n^{\:\! m} \, \right| \, = \, \left| \, \mathcal{F}_n^{\:\! m} \, \right| \, . \]
\end{proposition}

\begin{proof}
Define the function $\varphi \colon \mathcal{S}_n^{\:\! m} \rightarrow \mathcal{F}_n^{\:\! m}$ by the assignment $\sigma \mapsto (F_1, \ldots,F_m)\,$, where each $F_j$ is the (possibly empty) multiset $\{ \, \text{i}\left(\sigma_a\right) \, \mid \, a \in [n] \;\; \text{and} \;\; \sigma_a = j \, \}\,$.

We will first show that $\varphi$ maps $\mathcal{S}_n^{\:\! m}$ into $\mathcal{F}_n^{\:\! m}\,$. Since each $\sigma$ in $\mathcal{S}_n^{\:\! m}$ contains $n$ elements, it follows that $|F_1|+\cdots+|F_m|$ is equal to $n\,$. Also observe that if $\sigma_a$ is equal to $j\,$, then the value of $\text{i}(\sigma_a)$ is at most $| \, \{ \, b \in [n] \, \mid \, \sigma_b < \sigma_a \, \} \, |\,$.

We will now show injectivity. Assume $\sigma^1,\sigma^2$ are sequences in $\mathcal{S}_n^{\:\! m}$ such that $\varphi(\sigma^1)$ and $\varphi(\sigma^2)$ are both equal to $(F_1, \ldots, F_m)\,$. Observe that the elements of $F_m$ force the set $\{ \, a \in [n] \, \mid \, \sigma_a^i \, = \, m \, \}$ to be equal for $i=1,2\,$. Subsequently observe that the elements of $F_{m-1}$ force the sets $\{ \, a \in [n] \, \mid \, \sigma_a^i \, = \, m-1 \, \}$ to be equal for $i=1,2\,$, and so on. Hence, $\varphi$ is injective.

To show surjectivity, fix a fundamental sequence $(F_1,\ldots, F_m)$ in $\mathcal{F}_n^{\:\!m}\,$. We will construct a sequence $\sigma$ in $\mathcal{S}_n^{\:\! m}$ that maps to $(F_1,\ldots,F_m)$ via $\varphi\,$. Let $f_1^{\:\!m}\, \leq \, \ldots\,\leq\, f_{|F_m|}^{\:\! m}$ be the elements of $F_m\,$, and let $\sigma^0$ be the sequence in $\mathcal{S}_n^{\:\! m}$ containing only $1$'s. Find an $a \in [n]$ such that the cardinality of $\{ \, b \in [n] \,\; \vert \;\, b>a \; \text{ and } \; \sigma_b^0=1 \, \}$ equals $f_1^{\:\!m}\,$. Note, such an $a$ exists because of the restriction on the elements of $F_m$ in the definition of $\mathcal{F}_n^{\:\! m}\,$. Replace $\sigma_a^0$ with $m\,$, and call this new sequence $\sigma^1\,$. Next, find an $a \in [n]$ such that $\sigma_a^1$ is $1$ and the cardinality of $\{ \, b \in [n] \,\; \vert \;\, b>a \; \text{ and } \; \sigma_b^1=1 \, \}$ equals $f_2^{\:\!m}\,$. Replace $\sigma_a^1$ with $m\,$, and call this new sequence $\sigma^2\,$. Continuing similarly for the remaining elements of $F_m$ and for the elements of $F_{m-1}, \ldots,F_2$ respectively, observe that the multiset $\{ \, \text{i}(\sigma_a^{n-|F_1|}) \,\; \vert \;\, a \in [n] \, \text{ and } \, \sigma_a^{n-|F_1|}=j \, \}$ is equal to $F_j$ for all $j$ in $[m]\,$. As such, $\sigma^{n-|F_1|}$ is our desired $\sigma\,$.
\end{proof}
$\;$\\[-1.5em]

The following definition and corollary can help us clarify the result of Proposition \ref{prop:fundset_bij}.

\begin{definition}\label{def:F_n^m(k)}
If $m,n,k$ are nonnegative integers, define $\bm{\mathcal{F}_n^m(k)}$ in the following manner:
\[ \mathcal{F}_n^{\:\! m}(k) \, \coloneqq \, \left\lbrace \, (F_1,\ldots,F_m) \in \mathcal{F}_n^{\:\!m} \,\; \bigg\vert \;\, \underset{i\:\!=\:\!1}{\overset{m}{\textstyle \sum}} \, \underset{x \in F_i}{\textstyle \sum} x \, = \, k \, \right\rbrace \, .  \]
\end{definition}

Figure \ref{fig:F_n^m} contains some examples.

\begin{figure}[!htb]
\begin{center}
$\begin{array}{ccccc}
( \, \{0,0\} , \{2\} , \emptyset \,) \hspace{0.035in} & \hspace{0.035in} ( \,\{0,0\} , \emptyset  ,  \{2\} \, ) \hspace{0.035in} & \hspace{0.035in} (\,  \{0\} ,  \{1,1\}  , \emptyset \, ) \hspace{0.035in} & \hspace{0.035in} (\, \{0\} ,  \{1\} ,  \{1\} \, ) \hspace{0.035in} & \hspace{0.035in} ( \, \{0\} , \emptyset  , \{1,1\} \, )
\end{array}$
\end{center}
\caption{\label{fig:F_n^m} Fundamental sequences in $\mathcal{F}_3^{\:\!3}(2)$ for which $0 \not\in F_2$ and $0 \not\in F_3\,$.}
\end{figure}

\begin{corollary}\label{cor:F_n^m(k)}
If $m,n,k$ are nonnegative integers, then \[ \left| \, \mathcal{F}_n^{\:\! m}(k) \, \right| \, =\, \left| \, \{ \, \sigma \in \mathcal{S}_n^{\:\! m} \, \mid \, \emph{inv}(\sigma)=k \, \}\,\right| \, . \]
\end{corollary}

\begin{proof}
Restrict $\varphi$ from Proposition \ref{prop:fundset_bij} so its domain is $\{ \, \sigma \in \mathcal{S}_n^{\:\! m} \, \mid \, \text{inv}(\sigma)=k \, \}\,$.
\end{proof}

\subsection{Integer Partitions with Kinds}

Rather than using the conventional definition of integer partitions as finite series of nonincreasing positive integers, we have chosen a logically-equivalent and widely-utilized alternative which generalizes nicely to introduce integer partitions with kinds \cite{Gupta}.

\begin{definition}\label{def:partitions_kinds}
Let $k$ be a nonnegative integer. An \textbf{integer partition of $\bm{k}$} is a multiset of positive integers whose elements add to $k\,$. Let $m$ also be a nonnegative integer. An \textbf{integer partition of $\bm{k}$ with $\bm{m}$ kinds} is\[ ( \, P_1 \, , \, \ldots \, , \, P_m \, ) \;\, \text{ such that } \;\, \underset{i\:\!=\:\!1}{\overset{m}{\textstyle \sum}} \, \underset{x \in P_i}{\textstyle \sum} x \, = \, k \, , \] where each $P_i$ is a (possibly empty) multiset of positive integers. The set of all integer partitions of $k$ with $m$ kinds shall be denoted $\bm{\mathcal{P}_k^{\:\! m}}\,$.
\end{definition}

Figure \ref{fig:part_kinds} contains some examples. Observe the parallelism between Figures \ref{fig:F_n^m} and \ref{fig:part_kinds}.

\begin{figure}[!htb]
\begin{center}
$\begin{array}{ccccccccc}
(\, \{2\} , \emptyset \,) && ( \,\emptyset  ,  \{2\} \, ) && (\, \{1,1\}  , \emptyset \, ) && ( \,\{1\} ,  \{1\} \, ) &&  (\, \emptyset  , \{1,1\} \, )
\end{array}$
\end{center}
\caption{\label{fig:part_kinds} The integer partitions of $2$ with $2$ kinds.}
\end{figure}

\begin{proposition}\label{prop:kinds_fundamental}
Let $m,n$ be positive integers and $k$ be a nonnegative integer such that $k < n\,$. Then \[ \left| \, \mathcal{P}_k^{\:\! m} \, \right| \, = \, \left| \, \{ \, (F_1,\ldots,F_{m+1}) \in \mathcal{F}_n^{\:\! m+1}(k) \, \mid \, 0\not\in F_2\, ,\ldots, \, 0 \not\in F_{m+1} \, \} \, \right| \, .  \]
\end{proposition}

\begin{proof}
Define $\varphi \colon \mathcal{P}_k^{\:\! m} \rightarrow \{ \, (F_1,\ldots,F_{m+1}) \in \mathcal{F}_n^{\:\! m+1}(k) \, \mid \, 0\not\in F_2\, ,\ldots, \, 0 \not\in F_{m+1} \, \}$ via \[ (P_1,\ldots,P_m) \: \mapsto \: (F_1,P_1,\ldots,P_m) \, , \] where $F_1$ is a multiset of cardinality $n-|P_1|-\cdots-|P_m|$ containing only zeros. Note that since each $P_i$ contains only positive integers, the value of $|P_1|+\cdots+|P_m|$ is at most $k$ and $|F_1|$ is positive.

We will now show that $\varphi$ maps into the codomain. Observe that: the elements of $(F_1,P_1,\ldots,P_m)$ are $m+1$ multisets; the sum $|F_1|+|P_1|+\cdots+|P_m|$ equals $n\,$; the elements of $P_1,\ldots,P_m$ are nonzero; and the elements of $F_1,P_1,\ldots,P_m$ add to $k\,$. It remains to show that $(F_1,P_1,\ldots,P_m)$ satisfies the restriction on elements from the definition of $\mathcal{F}_n^{\:\! m+1}(k)\,$. Let $p$ equal $|P_1|+\cdots+|P_m|\,$. Since $P_1 \cup \cdots \cup P_m$ contains $p$ positive integers that sum to $k\,$, observe any element of $P_i$ could be at most $k-(p-1)$ in value. Since $k$ is less than $n\,$, it follows that \[ k-(p-1) \, = \, (k+1)-p \, \leq \, n-p \: . \] As such, every element in $P_1 \cup \cdots \cup P_m$ is in $[K_1]\,$, where $K_1$ is equal to $|F_1|\,$.

Finally, observe that bijectivity of $\varphi$ follows naturally from its rule of assignment.
\end{proof}
$\;$\\[-1.5em]

It may be worth reflecting on the set $\{ \, \sigma \in \mathcal{S}_n^{\:\! m+1} \,\; \vert \;\, \text{inv}(\sigma)=k \, \text{ and } \, \sigma_n=1 \, \}$ within the context of Propositions \ref{prop:fundset_bij} and  \ref{prop:kinds_fundamental}, respectively.

\section{$q$-multinomial Identities}

This section will focus on developing $q$-analogs of both classical and lesser-known multinomial identities. Our motivation is to provide a uniform treatment of these $q$-analogs utilizing the intuitive nature of the inversion statistic. It is the viewpoint of the authors that this approach is not only concise but also enables a deep understanding.

%In the majority of this section, we will be proving both classical binomial/multinomial identities and lesser known identities. The goal of this section is to provide a concise uniform treatment  
%
%In this section, we will acquaint ourselves with the act of generalizing binomial and multinomial identities into their corresponding $q$-analogs.

\subsection{Symmetry}

We will begin with the well-known $q$-analog to binomial symmetry, namely ${n \choose k} =  {n \choose {n-k}}\,$.

\begin{proposition}\label{prop:binom-sym}
If $n,k$ are nonnegative integers such that $n \geq k\,$, then \[ {n \choose k}_q \, = \, {n \choose {n-k}}_q \, . \]
\end{proposition}

\begin{proof}
Let $\mathcal{S}_n^{\:\! 2}(k,n-k)$ be the set of sequences of length $n$ whose elements are in $[2]$ with $k$ $2$'s, and refer to an arbitrary sequence in $\mathcal{S}_n^{\:\! 2}(k,n-k)$ by $\sigma= ( \, \sigma_1, \, \ldots, \, \sigma_n \,)\,$. For every $x \in [2]\,$, say that $\overline{x}$ equals $1$ when $x$ is $2$ and $\overline{x}$ equals $2$ when $x$ is $1\,$. Define a map \[ \varphi \colon \mathcal{S}_n^{\:\! 2}(k,n-k) \rightarrow \mathcal{S}_n^{\:\! 2}(n-k,k) \;\; \emph{by} \;\; ( \, \sigma_1, \, \ldots, \, \sigma_n \,) \mapsto ( \, \overline{\sigma_n}, \, \ldots, \, \overline{\sigma_1} \,) \, . \] Fix some $a \in [n]$ and consider $\sigma_a\,$. If $\sigma_a$ is 2 and i$\left(\sigma_a\right)$ is $c\,$, then the number of 1's that follow $\sigma_a$ in $\sigma$ must be $c\,$. By the definition of $\varphi\,$, notice the number of 2's preceding $\overline{\sigma_a}$ in $\varphi\left(\sigma\right)$ is also $c\,$. Hence, the numbers i$\left(\sigma_a\right)$ and r$\left(\overline{\sigma_a}\right)$ are equal. Should $\sigma_a$ be $1\,$, observe that i$\left(\sigma_a\right)$ and r$\left(\overline{\sigma_a}\right)$ are both zero. Further observing that $\varphi$ is bijective, the desired result follows from Proposition \ref{prop:received-induced}.
\end{proof}
$\;$\\[-1.5em]

We will now generalize to the well-known $q$-analog of multinomial symmetry.

\begin{proposition}\label{prop:multinom-sym}
If $m,n,k_1,\ldots,k_m$ are nonnegative integers such that $k_1+\cdots + k_m=n$ and $\pi$ is a permutation of $[m]\,$, then
\[ {n \choose {k_1, \, \ldots, \, k_m}}_q \, = \, {n \choose {k_{\pi(1)}, \, \ldots, \, k_{\pi(m)}}}_q \, . \]
\end{proposition}

\begin{proof}
Refer to an arbitrary sequence in $\mathcal{S}_n^{\:\! m}(k_1, \ldots, k_m)$ by $\sigma= ( \, \sigma_1, \, \ldots, \, \sigma_n \,)\,$, and define a map
\[ \theta\colon \mathcal{S}_n^{\:\! m}(k_1,\ldots,k_i,k_{i+1},\ldots,k_m) \rightarrow \mathcal{S}_n^{\:\! m} (k_1,\ldots,k_{i+1},k_i,\ldots,k_m) \] such that $\theta(\sigma)_a$ equals $\sigma_a$ when $\sigma_a$ is neither $i$ nor $i+1\,$. It follows that
\begin{align*}
\displaystyle\sum_{\sigma_a > i+1} \text{i}(\sigma_a) \, &= \, \displaystyle\sum_{\theta(\sigma)_a > i+1} \text{i}\left(\theta(\sigma)_a\right) \, , & \displaystyle\sum_{\sigma_a < i} \text{i}(\sigma_a) \, &= \, \displaystyle\sum_{\theta(\sigma)_a < i} \text{i}\left(\theta(\sigma)_a\right) \, .
\end{align*}
For the subsequence of $\sigma$ for which $\sigma_a$ is equal to $i$ or $i+1\,$, let $\theta$ act on that subsequence analogously to $\varphi$ in Proposition \ref{prop:binom-sym}.  It follows that \[ \displaystyle\sum_{\sigma_a \in \{i,i+1\}} \text{i}(\sigma_a) \, = \, \displaystyle\sum_{\theta(\sigma)_a \in \{i,i+1\}} \text{i}\left(\theta(\sigma)_a\right) \, . \]
By Proposition \ref{prop:received-induced}, we have that inv($\sigma$) equals inv($\theta(\sigma)$)\,.

Observe that this Proposition has been shown for $\pi$ that are of the form of a simple transposition. Given that any permutation is a composition of simple transpositions, we have our desired result for any permutation $\pi\,$.
\end{proof}
$\;$\\[-1.5em]

\subsection{Pascal's Identity}

We will now consider Pascal's Identity in order to develop a well-known $q$-analog,
\[ {n \choose {k_1,\ldots,k_m}} \, = \, {{n-1} \choose {k_1-1,\ldots, k_m}} + {{n-1} \choose {k_1, k_2-1,\ldots, k_m}} + \cdots + {{n-1} \choose {k_1,\ldots,k_{m}-1}} \,. \]
Interpreting ${n \choose {k_1,\ldots,k_m}}$ as the number of sequences in $\mathcal{S}_n^{\:\! m}(k_1, \ldots, k_m)\,$, then ${{n-1} \choose {k_1-1,\ldots,k_m}}$ counts such sequences that end in a $1\,$, ${{n-1} \choose {k_1,k_2-1,\ldots,k_m}}$ counts such sequences that end in a $2\,$, and so on.

\begin{proposition}\label{prop:pascal-multinom}
If $m,n,k_1,\ldots,k_m$ are nonnegative integers such that $k_1+\cdots + k_m=n\,$, then ${n \choose {k_1,\ldots,k_m}}_q$ is equal to
\[ q^{k_2+\cdots+k_m}{{n-1} \choose {k_1-1,\ldots, k_m}}_q + q^{k_3+\cdots+k_m}{{n-1} \choose {k_1, k_2-1,\ldots, k_m}}_q + \cdots + {{n-1} \choose {k_1,\ldots,k_{m}-1}}_q \,. \]
\end{proposition}

\begin{proof}
Interpret ${n \choose {k_1,\ldots,k_m}}_q$ as the generating function for inversions on $\mathcal{S}_n^{\:\! m}(k_1, \ldots, k_m)\,$. For such sequences that end in a $1\,$, note that $k_2+\cdots+k_m$ inversions will be received by that $1\,$. Thus, the product of $q^{k_2+\cdots+k_m}$ and ${{n-1} \choose {k_1-1,\ldots,k_m}}_q$ accounts precisely for the inversions of sequences that end in a $1\,$. The argument is similar for the remaining terms of our desired sum.
\end{proof}
$\;$\\[-1.5em]

Note that applying Proposition \ref{prop:multinom-sym} to Proposition \ref{prop:pascal-multinom} yields $m!$ different articulations of the $q$-analog to Pascal's Identity. For the case $m=2\,$, Figure \ref{fig:Pascal-binom} contains the resulting $2!$ articulations.

\begin{figure}[!htb]
\[ q^{k_2}{{n-1} \choose {k_1-1,k_2}}_q+{{n-1} \choose {k_1,k_2-1}}_q \hspace{.65in} {{n-1} \choose {k_1-1,k_2}}_q+q^{k_1}{{n-1} \choose {k_1,k_2-1}}_q \]
\caption{\label{fig:Pascal-binom} The two articulations of ${n \choose {k_1,k_2}}_q$ via the $q$-analog of Pascal's Identity.}
\end{figure}

\subsection{A Useful Lemma}

We will now introduce a lemma that will supply a framework of thinking to address a number of the Propositions that follow within this section.

\begin{lemma}\label{lem:chunk}
Let $m,n,t$ be nonnegative integers; let $p_1 \, , \ldots \, , \, p_t$ be nonnegative integers such that $p_1+\cdots+p_t = n\,$; for all $(i,j)$ in $[t]\times [m]\,$, let $\ell_{i,j}$ be nonnegative integers such that $\ell_{i,1}+\cdots+\ell_{i,m} = p_i \,$; for all $i$ in $[t]\,$, let $d_i$ equal $p_1+\cdots+p_i\,$; for any $\sigma$ in $\mathcal{S}_n^{\:\! m}$ and for all $u \in [t]\,$, let $s_u(\sigma)$ be $\left( \, \sigma_{d_{u-1}+1} \, , \, \ldots \, , \, \sigma_{d_u} \, \right)\,$; and let \[ T \, \coloneqq \, \left\lbrace \, \sigma \in \mathcal{S}_n^{\:\! m} \, \mid \, s_u(\sigma) \in S_{p_u}^{\:\! m}\left(\ell_{u,1} \, , \, \ldots \, , \, \ell_{u,m} \right)\, , \;\, \forall \: u \in [t] \, \right\rbrace\, . \]
Then,
\[ \sum_{\sigma \in T} q^{\emph{inv}(\sigma)} \: = \: {q{\Large \strut}}^{\Big(\, \underset{u=1}{\overset{t}{\scriptstyle \sum}} \, \underset{v=2}{\overset{m}{\scriptstyle \sum}} \, \ell_{u,v} \, \big( \, \underset{i=u+1}{\overset{t}{\scriptstyle \sum}} \, \underset{j=1}{\overset{v-1}{\scriptstyle \sum}} \, \ell_{i,j} \, \big) \, \Big) } { p_1 \choose {\ell_{1,1},\ldots,\ell_{1,m}}}_q \cdots {p_t \choose {\ell_{t,1},\ldots,\ell_{t,m}}}_q \, . \]
\end{lemma}

\begin{proof}
Let $X_1$ be equal to $[d_1]\,$; for all $2\leq u\leq t\,$, let $X_u$ be equal to $[d_u]\setminus [d_{u-1}]\,$; and interpret the left-hand side of the equality in the lemma statement as the generating function for inversions on $T\,$.

Observe that for a sequence $\sigma$ in $T\,$, ordered pairs $(a,b)$ associated with $\text{inv}(\sigma)$ are of exactly one of the following forms: $a,b$ are both in the same $X_u\,$; and $a,b$ are not both in the same $X_u\,$. Note that for all $u$ in $[t]\,$, ${p_u \choose {\ell_{u,1},\ldots,\ell_{u,m}}}_q$ accounts for ordered pairs $(a,b)$ associated with inversions such that $a,b$ are both in $X_u\,$. It remains to consider ordered pairs $(a,b)$ associated with inversions such that $a,b$ are not both in the same $X_u\,$.

Fix some $u$ in $[t]$ and some $v$ in $[m]\setminus \{1\}\,$. Observe that $\ell_{u,v}$ is the number of elements $\sigma_a$ such that $a$ is in $X_u$ and $\sigma_a$ equals $v\,$. Furthermore, ${\scriptstyle \sum} \, {\scriptstyle \sum} \, \ell_{i,j}$ accounts for the number of elements $\sigma_b$ such that $b$ is greater than $d_u$ and $\sigma_b$ is less than $v\,$. Hence, $\ell_{u,v} \left( \, {\scriptstyle \sum} \, {\scriptstyle \sum} \, \ell_{i,j}  \, \right)$ accounts for ordered pairs $(a,b)$ associated with inversions where $a$ is in $X_u\,$, $\sigma_a$ equals $v\,$, and $b$ in not in $X_u\,$. Summing over all $\ell_{u,v}\,$, the factor of $q^{\scriptstyle \sum (\cdots)}$ accounts for ordered pairs $(a,b)$ associated with inversions such that $a,b$ are not both in the same $X_u\,$.
\end{proof}

\subsection{Diagonal Sum Identity}

We will now consider the Diagonal Sum Identity in order to develop a well-known $q$-analog,
\[ {n \choose {k_1,\ldots,k_m}} \, = \, \sum_{i=0}^{k_1} \, \sum_{j=2}^m {{n-i-1} \choose {k_1-i,k_2,\ldots,k_j-1,\ldots, k_m}}\, . \]
Interpreting ${n \choose {k_1,\ldots,k_m}}$ as the number of sequences in $\mathcal{S}_n^{\:\! m}(k_1,\ldots,k_m)\,$, then the expression ${{n-i-1} \choose {k_1-i,k_2,\ldots,k_j-1,\ldots,k_m}}$ counts such sequences that end in a $j$ followed by $i$ 1's.

\begin{proposition}\label{prop:diag}
If $m,n,k_1,\ldots,k_m$ are nonnegative integers such that $k_1+\cdots + k_m=n\,$, then \[ {n \choose {k_1,\ldots,k_m}}_q \, = \, \sum_{i=0}^{k_1} \, \sum_{j=2}^m \: q{\large \strut}^{\Big(\, (n-k_1)\:\!i+ \underset{v = j+1}{\overset{m}{\scriptstyle \sum}} k_v \, \Big)}  {{n-i-1} \choose {k_1-i,k_2,\ldots,k_j-1,\ldots, k_m}}_q\, . \]
\end{proposition}

\begin{proof}
Interpret ${n \choose {k_1,\ldots,k_m}}_q$ as the generating function for inversions on $\mathcal{S}_n^{\:\! m}(k_1, \ldots , k_m)\,$. We will apply Lemma \ref{lem:chunk} to achieve our desired result on the terms of the right-hand side. We will do so by letting: $p_1$ equal $n-i-1\,$; $\ell_{1,1}$ equal $k_1-i\,$; $\ell_{1,j}$ equal $k_j-1\,$; $\ell_{1,v}$ equal $k_v$ for all $v$ in $[m] \setminus \{1,j\}\,$; $p_2$ and $\ell_{2,j}$ equal $1\,$; $p_3$ and $\ell_{3,1}$ equal $i\,$.

Since this is our first application of Lemma \ref{lem:chunk}\,, we will provide explicit details regarding the Lemma's implementation. Note that: ${n-i-1 \choose {k_1-i,\ldots,k_m}}_q$ accounts for ordered pairs $(a,b)$ associated with inversions such that $a,b$ are both in the same $X_u\,$; there are $(n-k_1-1)\:\!i+\sum k_v$ ordered pairs $(a,b)$ associated with inversions such that $a$ is in $X_1$ and $b$ is not in $X_1\,$; and there are $i$ ordered pairs $(a,b)$ associated with inversions such that $a$ is in $X_2$ and $b$ is not in $X_2\,$. Hence, our factor of $q^{\scriptstyle(\cdots)}$ corresponds precisely with that in Lemma \ref{lem:chunk}\,.
\end{proof}

\subsection{Vandermonde's Identity}

Next, we will consider Vandermonde's Identity to develop a well-known $q$-analog,
\[ {n_1+n_2 \choose {k_1,\ldots,k_m}} \, = \, \sum_{\substack{r_1+\cdots+r_m \:\!= \:\!n_1 \\[1pt] 0 \leq r_i \leq k_i}} {n_1 \choose {r_1,\ldots,r_m}}{n_2 \choose {k_1-r_1, \ldots, k_m-r_m}} \, . \]
Interpreting ${{n_1+n_2} \choose {k_1,\ldots,k_m}}$ as the number of sequences in $\mathcal{S}_{n_1+n_2}^{\:\! m}(k_1, \ldots , k_m)\,$, then each term of the sum accounts for the sequences whose first $n_1$ elements contains exactly $r_1$ 1's, $\ldots$, $r_m$ m's.

\begin{proposition}\label{prop:Vandermonde}
If $m,n_1,n_2,k_1,\ldots,k_m$ are nonnegative integers such that $k_1+\cdots + k_m$ equals $n_1+n_2\,$, then \[ {n_1+n_2 \choose {k_1,\ldots,k_m}}_q \, = \, \sum_{\substack{r_1+\cdots+r_m \:\!= \:\!n_1 \\[1pt] 0 \:\! \leq \:\! r_i \:\! \leq \:\! k_i}} q{\large \strut}^{\Big(\, \underset{v = 2}{\overset{m}{\scriptstyle \sum}} f(r_v) \, \Big) } {n_1 \choose {r_1,\ldots,r_m}}_q{n_2 \choose {k_1-r_1, \ldots, k_m-r_m}}_q \] where $f(r_v) \, = \, r_v \:\! \underset{j =1}{\overset{v-1}{\sum}} (k_j-r_j)$ for every $v \in [m]\,$.
\end{proposition}

\begin{proof}
Interpret ${{n_1+n_2} \choose {k_1,\ldots,k_m}}_q$ as the generating function for inversions on $\mathcal{S}_{n_1+n_2}^{\:\! m}(k_1, ..., k_m)\,$. We will apply Lemma \ref{lem:chunk} to achieve our desired result on the terms of the right-hand side. We will do so by letting: $p_1$ equal $n_1\,$; $\ell_{1,v}$ equal $r_v$ for all $v \in [m]\,$; $p_2$ equal $n_2\,$; and $\ell_{2,v}$ equal $k_v-r_v$ for all $v \in [m]\,$.  

Note that: the $q$-multinomial coefficients on the right-hand side account for ordered pairs $(a,b)$ associated with inversions such that $a,b$ are both in the same $X_u\,$; and $q^{\scriptstyle \sum f(r_v)}$ accounts for ordered pairs $(a,b)$ associated with inversions such that $a,b$ are not both in the same $X_u\,$.
\end{proof}
$\;$\\[-1.5em]

For a well-known generalization of Vandermonde's Identity, we will provide a $q$-analog.

\begin{proposition}\label{prop:genVandermone}
If $m,n_1,...,n_t,k_1,...,k_m$ are nonnegative integers such that $k_1+\cdots + k_m$ is equal to $n_1+\cdots+n_t\,$, then \[ {{n_1+\cdots+n_t} \choose {k_1,\ldots,k_m}}_q = \hspace{-3pt} \sum_{\substack{r_{u,1}+\cdots+r_{u,m} \:\!= \:\!n_u \\[1pt] r_{1,v}+\cdots+r_{t,v} \:\! = \:\! k_v \\[1pt] 0 \:\! \leq \:\! r_{u,v}}} q{\large \strut}^{\Big(\, \underset{u=1}{\overset{t}{\scriptstyle \sum}}\,\underset{v=2}{\overset{m}{\scriptstyle \sum}} f(r_{u,v}) \, \Big) } {n_1 \choose {r_{1,1}\:\!,\ldots,r_{1,m}}}_q \cdots {n_t \choose {r_{t,1}\:\!, \ldots, r_{t,m}}}_q \] where $f(r_{u,v}) \, = \, r_{u,v} \:\! \underset{i\:\!=\:\!u+1}{\overset{t}{\sum}} \, \underset{j\:\!=\:\! 1}{\overset{v-1}{\sum}} r_{i,j}$ for every $(u,v) \in [t]\times[m]\,$.
\end{proposition}

\begin{proof}
Consider $\mathcal{S}_{n_1+\cdots+n_s}^{\:\! m}(k_1, \ldots, k_m)\,$, and interpret ${{n_1+\cdots+n_s} \choose {k_1,\ldots,k_m}}_q$ as the generating function for inversions on this set of sequences. Letting $p_i$ equal $n_i$ for all $i \in [t]$ and $\ell_{u,v}$ equal $r_{u,v}$ for all $(u,v)$ in $[t]\times[m]\,$, the application of Lemma \ref{lem:chunk} is direct.
\end{proof}

\subsection{Chu Shih-Chieh (Zhu Shijie)'s Identity}

We will now consider the well-known Chu Shih-Chieh's Identity,
\[ {n \choose {k_1,\ldots,k_m}} \, = \, \sum_{r=0}^{n-k_1} \, \sum_{\substack{r_2+\cdots+r_m \:\!= \:\!r \\[1pt] 0 \leq r_j \leq k_j}} {r \choose {0,r_2,\ldots,r_m}}{n-r-1 \choose {k_1-1,k_2-r_2, \ldots, k_m-r_m}} \, . \]
Interpreting ${{n} \choose {k_1,\ldots,k_m}}$ as the number of sequences in $\mathcal{S}_n^{\:\! m}(k_1,\ldots, k_m)\,$, then each term of the sum accounts for the sequences $( \, \sigma_1, \, \ldots, \, \sigma_n \,)$ such that: $\sigma_{r+1}$ equals 1; $( \, \sigma_1, \, \ldots, \, \sigma_r \,)$ is a sequence with $r_2$ 2's, $\ldots$, $r_m$ m's; and $( \, \sigma_{r+2}, \, \ldots, \, \sigma_n \,)$ is a sequence with $k_1-1$ 1's, $k_2-r_2$ 2's, $\ldots$, $k_m-r_m$ m's.

A lesser-known, albeit natural, generalization follows.

\begin{proposition}\label{prop:Chu}
If $m,n, k_1, \ldots, k_m$ are nonnegative integers such that $k_1+\cdots+k_m$ is equal to $n\,$, then
\[ {n \choose {k_1,\ldots,k_m}} \, = \, \sum_{\substack{E \subseteq [n] \\[1pt] |E|=k_1}} \, \sum_{\substack{r_{i,2}+\cdots+r_{i,m} \:\!= \:\!n_i \\[1pt] r_{1,j}+ \cdots + r_{s,j} \:\! = \:\! k_j \\[1pt]  0 \leq r_{i,j}}} {{n_1} \choose {0,r_{1,2},...,r_{1,m}}} \cdots {n_s \choose {0,r_{s,2}\:\!, ..., r_{s,m}} } \]
where $E \, = \, \{ \, e_1, \, \ldots, \, e_{k_1} \, \}$ with $e_i < e_{i+1}$ for every $1\leq i \leq k_1-1\,$; $s$ is equal to $k_1+1\,$; $n_1$ equals $e_1-1\,$; $n_i$ equals $e_i-e_{i-1}-1$ for every $2\leq i \leq k_1\,$; and $n_s$ equals $n-e_{k_1}\,$.
\end{proposition}

\begin{proof}
Interpret ${n \choose {k_1,\ldots,k_m}}$ as the number of sequences in $\mathcal{S}_n^{\:\! m}(k_1,\ldots, k_m)\,$. Each term of the sum accounts for the sequences $( \, \sigma_1, \, \ldots, \, \sigma_n \,)$ such that: $\sigma_i$ equals $1$ if and only if $i$ is in $E\,$; the subsequence $( \, \sigma_1, \, \ldots, \, \sigma_{e_1-1} \,)$ contains $r_{1,2}$ 2's, $\ldots$, $r_{1,m}$ m's; for every $2\leq i \leq k_1\,$, the subsequence $( \, \sigma_{e_{i-1}+1}, \, \ldots, \, \sigma_{e_i-1} \,)$ contains $r_{i,2}$ 2's, $\ldots$, $r_{i,m}$ m's; and the subsequence $( \, \sigma_{e_{s-1}+1}, \, \ldots, \, \sigma_n \,)$ contains $r_{s,2}$ 2's, $\ldots$, $r_{s,m}$ m's.
\end{proof}
$\;$\\[-1.5em]

A $q$-analog of Proposition \ref{prop:Chu} follows.

\begin{proposition}\label{prop:Chu_q}
If $m,n, k_1, \ldots, k_m$ are nonnegative integers such that $k_1+\cdots+k_m$ is equal to $n\,$, then ${n \choose {k_1,\ldots,k_m}}_q$ is equal to
\[ \sum_{\substack{E \subseteq [n] \\[1pt] |E|=k_1}} \, \sum_{\substack{r_{u,2}+\cdots+r_{u,m} \:\!= \:\!n_u \\[1pt] r_{1,v}+ \cdots + r_{s,v} \:\! = \:\! k_v \\[1pt]  0 \leq r_{u,v}}} q{\large \strut}^{\Big(\underset{u=1}{\overset{s}{\scriptstyle \sum}} \, \underset{v=2}{\overset{m}{\scriptstyle \sum}} \, f(r_{u,v}) \, \Big) } {{n_1} \choose {0,r_{1,2},...,r_{1,m}}}_q \cdots {n_s \choose {0,r_{s,2}\:\!, ..., r_{s,m}} }_q \]
where $E \, = \, \{ \, e_1, \, \ldots, \, e_{k_1} \, \}$ with $e_i < e_{i+1}$ for every $1\leq i \leq k_1-1\,$; $s$ is equal to $k_1+1\,$; $n_1$ equals $e_1-1\,$; $n_i$ equals $e_i-e_{i-1}-1$ for every $2\leq i \leq k_1\,$;  $n_s$ equals $n-e_{k_1}\,$; and $f(r_{u,v}) \, = \, r_{u,v} \:\! \Big( \, k_1-u+1 +  \underset{i\:\!=\:\!u+1}{\overset{s}{\sum}} \,\:\!  \underset{j\:\!=\:\! 2}{\overset{v-1}{\sum}} r_{i,j}\, \, \Big)$ for every $(u,v) \in [s]\times[m]$\,.
\end{proposition}

\begin{proof}
Interpret ${n \choose {k_1,\ldots,k_m}}_q$ as the generating function for inversions on $\mathcal{S}_n^{\:\! m}(k_1, \ldots, k_m)\,$. We will apply Lemma \ref{lem:chunk} to achieve our desired result on the sequences $\sigma$ associated with the terms of the right-hand side. We will do so by letting: $p_{2k}$ equal $1$ for all $k$ in $[k_1]\,$; $\ell_{2k,1}$ equal $1$ for all $k$ in $[k_1]\,$; $p_{2k-1}$ equal $n_k$ for all $k$ in $[s]\,$; $\ell_{2k-1,1}$ equal zero for all $k$ in $[s]\,$; and $\ell_{2k-1,v}$ equal $r_{k,v}$ for all $k$ in $[s]$ and for all $v$ in $[m]\setminus{1}\,$.

Observe that: the $q$-multinomial coefficients on the right-hand side account for ordered pairs $(a,b)$ associated with inversions such that $a,b$ are both in $X_{2k-1}$ for some $k$ in $[s]\,$; when $a$ is in $X_{2k}$ for some $k$ in $[k_1]\,$, then $\sigma_a$ equals $1$ and no ordered pairs $(a,b)$ are associated with inversions; $r_{u,v}\left(k_1-u+1\right)$ accounts for ordered pairs $(a,b)$ associated with inversions such that $a$ is in $X_{2k-1}$ for some $k$ in $[s]\,$, $\sigma_a$ equals $v\,$, and $b$ is in $X_{2j}$ for some $j$ in $[k_1]\,$; and $r_{u,v}\left( \, {\scriptstyle \sum} \, {\scriptstyle \sum} \, r_{i,j}  \, \right)$ accounts for ordered pairs $(a,b)$ associated with inversions such that $a$ is in $X_{2k-1}$ for some $k$ in $[s]\,$, $\sigma_a$ equals $v\,$, and $b$ is in $X_{2j-1}$ for some $j$ in $[s]$ where $j$ is greater than that $k\,$. \end{proof}

\subsection{``Apartment Complex" Identity}

The following identity was inspired from an identity contained in \cite{Old}. Consider a hypothetical scenario with an apartment complex whose buildings will contain exactly one unit per floor. Assume there are to be $n_1$ buildings, with $n_2$ of them receiving a second floor. Exactly $k$ of the units will be rented.

\[ {n_1 \choose n_2}{{n_1+n_2} \choose k} \, = \, \sum_{k_1+k_2=k} \, {n_1 \choose k_1}{n_1 \choose {n_1-n_2,k_2,n_2-k_2}} \, . \]

The complex owner could first choose which $n_2$ of the $n_1$ buildings will receive a second floor, and then $k$ tenants could choose which of the $n_1+n_2$ units to rent. Alternatively, for all $k_1$ in between $0$ and $k\,$, the owner could first rent out $k_1$ of the $n_1$ first floor units, and then of the $n_1$ buildings: $n_1-n_2$ buildings could receive no second floor; $k_2$ of them could receive a second floor that is rented; and $n_2-k_2$ could receive a second floor that is unrented. This naturally generalizes as follows.

\begin{proposition}\label{prop:Apartment}
If $n_1,\ldots,n_j, k$ are nonnegative integers such that $n_j \leq \cdots \leq n_1$ and $k \leq n_1 + \cdots + n_j\,$, then 
\[ \left( \, \prod_{i=2}^{j}{n_{i-1} \choose n_i} \, \right) {{n_1+\cdots+n_j} \choose k} \, = \sum_{k_1+\cdots+k_j=k} \,  {n_1 \choose k_1}\: \prod_{i=2}^{j}{n_{i-1} \choose {n_{i-1}-n_{i},n_{i}-k_{i},k_{i}}} \, . \]
\end{proposition}

\begin{proof}
For every $2 \leq i \leq j\,$, let $S_{i-1}$ be the set $S_{n_{i-1}}^{\:\! 2}(n_{i-1}-n_i,n_i)\,$. Also let $S_j$ be the set $S_{n_1+\cdots+n_j}^{\:\! 2}(n_1+\cdots+n_j-k, k)\,$. For every $k_1+\cdots+k_j$ equal to $k\,$, let: $T_1^{(k_1,\ldots,k_j)}$ be the set $S_{n_1}^{\:\! 2}(n_1-k_1, k_1)\,$; and for every $2 \leq i \leq j\,$, let $T_i^{(k_1,\ldots,k_j)}$ be the set $S_{n_{i-1}}^{\:\! 3} (n_{i-1}-n_i,n_i-k_i,k_i )\,$.

Define \[ \varphi \colon \prod_{i=1}^j S_i \rightarrow \coprod_{k_1+\cdots+k_j=k} \left( \, \prod_{i=1}^j T_i^{(k_1,\ldots,k_j)} \, \right) \;\; \text{via} \;\; (\sigma^1,\ldots,\sigma^j) \mapsto (\tau^1 , \ldots , \tau^j) \] in the following way. For every $1\leq i \leq j-1\,$, let $C_{i} \, = \, \{ \, s \in [n_{i}] \, \mid \, \sigma_s^{i} \, = \, 2 \, \}\,$. Express $C_i$ as $\{ \, c_{i,1}\,, \ldots , c_{i,n_{i+1}} \, \}$ where $c_{i,p} < c_{i,p+1}$ for every $1 \leq p \leq n_{i+1}-1\,$. Further, let $N_i$ be equal to $n_1+\cdots +n_i\,$. Finally, for every $1\leq i \leq j-1\,$, let
\begin{align*}
\tau_s^1 \, &= \, \sigma_s^j \: ,  \\[3pt]
\tau_s^{i+1} \, &= \, \begin{cases} 1 \; \text{ if } \; \sigma_s^i \, = \, 1 \: ,  \\[2pt] 2 \; \text{ if } \; \sigma_s^i \, = \, 2 \; \text{ and } \; \sigma_{N_i+p}^{\:\!j} \, = \, 1 \; \text{ where } \; s\, = \, c_{i,p} \: , \\[2pt] 3 \; \text{ if } \; \sigma_s^i \, = \, 2 \; \text{ and } \; \sigma_{N_i+p}^{\:\!j} \, = \, 2 \; \text{ where } \; s \, = \, c_{i,p} \: . \end{cases}
\end{align*}
The desired result follows from observing that $\varphi$ is bijective.
\end{proof}
$\;$\\[-1.5em]

Our $q$-analog of Proposition \ref{prop:Apartment} follows.

\begin{proposition}\label{prop:q-Apartment}
If $n_1,\ldots,n_j, k$ are nonnegative integers such that $n_j \leq \cdots \leq n_1$ and $k \leq n_1 + ... + n_j\,$, then 
\[ \left( \, \prod_{i=2}^{j}{n_{i-1} \choose n_i}_q \, \right) {{n_1+ \cdots +n_j} \choose k}_q \, = \sum_{k_1+\cdots+k_j=k} q^{f(K)}  {n_1 \choose k_1}_q \: \prod_{i=2}^{j}{n_{i-1} \choose {n_{i-1}-n_{i},n_{i}-k_{i},k_{i}}}_q \] where $f(K) \, = \, \underset{u\:\!=\:\! 1}{\overset{j-1}{\sum}} k_u \, \Big( \,  \underset{i\:\!=\:\!u+1}{\overset{j}{\sum}} (n_i-k_i) \, \Big)$ for every $K$ equal to $( \, k_1, \ldots,k_j \, )\,$.
\end{proposition}

\begin{proof}
We will utilize the notation of Proposition \ref{prop:Apartment} and interpret the $q$-analogs within this identity as generating functions for the inversion statistic on sequences.

We will begin by accounting for the inversions associated with ${{n_1+ \cdots + n_j} \choose k}_q\,$, namely the inversions associated with $\sigma^j\,$. We will apply Lemma \ref{lem:chunk} by letting (for all $i$ in $[j]\,$): $p_i$ equal $n_i\,$; $\ell_{i,1}$ equal $n_i-k_i\,$; and $\ell_{i,2}$ equal $k_i\,$. Note that for every $i$ in $[j]\,$, the ordered pairs $(a,b)$ associated with inv$\left( \sigma^j \right)$ such that $a,b$ are both in $X_i\,$ are accounted for by \begin{align*}
\text{inv}\left( \tau^1 \right) \, , &\;\;\text{when } \; i=1\, ; \\[2pt] \sum_{\tau_s^i \:\! = \:\! 2} \text{r}\left( \tau_s^i \right) \, , & \;\;\text{when } \; i\geq 2\, .
\end{align*} Also note that $q^{f(K)}$ accounts for ordered pairs $(a,b)$ associated with inv$\left( \sigma^j\right)$ such that $a,b$ are not both in $X_i$ for some $i$ in $[j]\,$.

We will now account for inversions associated with $\prod {n_{i-1} \choose n_i}_q\,$. Observe that for every $2 \leq i \leq j\,$, \[ \text{inv}\left( \sigma^{i-1} \right) \, = \, \sum_{\sigma_s^{i-1} \:\! = \:\! 1} \text{r}\left( \sigma_s^{i-1} \right) \, = \, \sum_{\tau_s^i \:\! = \:\! 1} \text{r} \left( \tau_s^i \right) \, . \]

The desired result follows as an application of Corollary \ref{cor:received-induced}. \end{proof}
$\;$\\[-1.5em]

Notice that developing a deep enumerative understanding of the original ``apartment complex" identity in terms of sequences enabled us to develop a corresponding $q$-analog. It is the viewpoint of the authors that a complete grasp of the enumerative combinatorics of any binomial or multinomial identity supplements the development of a $q$-analog generalization.

\section{Integer Partitions and Galois Numbers} 

In this section, we will introduce the major index statistic and generalized Galois numbers. Ultimately, we will develop a theorem that reveals a connection between the coefficients of generalized Galois numbers and integer partitions with kinds.

\subsection{Major Index Statistic}

We will now formally define the well-known major index statistic. From our experience, the major index statistic more naturally and elegantly conveys the results sought in this section, which were proving to be a cumbersome endeavor using the inversion statistic. 

\begin{definition}\label{def:majors}
If $m,n$ are nonnegative integers and $\sigma = ( \, \sigma_1, \, \ldots, \, \sigma_n \,)$ is a sequence whose elements are in $[m]\,$, then \[ \emph{maj}(\sigma) \, \coloneqq \, \sum_{\substack{a \in [n-1] \\ \sigma_a \:\! > \:\! \sigma_{a+1}}} a \: . \] The value of \emph{maj(}$\sigma$\emph{)} shall be referred to as the \textbf{major index of $\bm{\sigma}$}\,.
\end{definition}

Figure \ref{fig:sequences1} contains some examples. Observe the parallelism between Figures \ref{fig:sequences} and \ref{fig:sequences1}.

\begin{figure}[!htb]
\begin{center}
$\begin{array}{ccccc}
2211 & \hspace{.25in} & 2121 & \hspace{.25in} & 2112 \\
\text{maj}(\sigma)=2 && \text{maj}(\sigma)=4 && \text{maj}(\sigma)=1 \\[16pt]
1221 && 1212 && 1122 \\
\text{maj}(\sigma)=3 && \text{maj}(\sigma)=2 && \text{maj}(\sigma)=0
\end{array}$
\end{center}
\caption{\label{fig:sequences1} All sequences of length $4$ with two $2$s and two $1$s.}
\end{figure}

The observed parallelism between Figures \ref{fig:sequences} and \ref{fig:sequences1} is in fact not a coincidence. MacMahon showed in \cite{Mac} that when considering the set of sequences $\mathcal{S}_n^{\:\! m}(k_1,\ldots,k_m)\,$, the generating functions for major index and inversions are equal: a fact which is now well-known. Stated precisely, if $m,n,k_1,\ldots,k_m$ are nonnegative integers such that $k_1+\cdots+k_m=n\,$, then

\begin{equation}\label{eq:MacMahon} {n \choose k_1,\ldots, k_m}_q \, = \, \sum_{\sigma \in \mathcal{S}_n^{\:\! m}(k_1, \ldots,k_m)} q^{\text{inv}(\sigma)} \, = \, \sum_{\sigma \in \mathcal{S}_n^{\:\! m}(k_1, \ldots,k_m)} q^{\text{maj}(\sigma)} \, . \end{equation}
$\;$\\[-1.5em]

The following two lemmas and corollary will develop additional familiarity with the major index statistic while also proving vital in the later theorem.

\begin{lemma}\label{lem:maj_XXX}
Let $m,n,k$ be nonnegative integers such that $n-m+1 \geq k+1\,$, 
\begin{align*}
\mathcal{M}_n^{\:\! m}(k) \, \coloneqq& \, \{ \, \sigma \in \mathcal{S}_n^{\:\! m} \, \mid \, \emph{maj}(\sigma)=k \, \} \, , \\[2pt] A_i \, =& \, \{ \, \sigma \in \mathcal{M}_n^{\:\! m+1}(k) \, \mid \, \sigma_{n-m+i}=\sigma_{n-m+i+1} \, \} \; \emph{ when } \; 1 \leq i \leq m-1 \, , \\[2pt] A_m \, =& \, \{ \, \sigma \in \mathcal{M}_n^{\:\! m+1}(k) \, \mid \, \sigma_n=m+1 \, \} \,  . 
\end{align*}
Then, 
\[ \mathcal{M}_n^{\:\! m+1}(k) 
\, \setminus \, \bigcup_{i \in [m]} A_i \: = \: \{ \, \sigma \in \mathcal{M}_n^{\:\! m+1}(k) \, \mid \, \sigma_{k+1}=1\;\; \emph{and} \;\; \forall \, j \in [m] \,,\; \sigma_{n-m+j}=j \,  \} \, . \]
\end{lemma}

\begin{proof}
Let $\sigma$ be in $\mathcal{M}_n^{\:\! m+1}(k) \setminus \cup A_i$ and $\omega$ be the subsequence of $\sigma$ containing its last $m$ elements, namely $( \, \sigma_{n-m+1}, \, \ldots , \, \sigma_n \, )\,$. Since $n-m+1$ must be at least $k+1$ in value and $\text{maj}(\sigma)$ is equal to $k\,$, the subsequence $\omega$ must be nondecreasing. In addition, since $\sigma$ is not in $\cup A_i\,$, the subsequence $\omega$ must be strictly increasing and not end in $m+1\,$. Given that the length of $\omega$ is $m\,$, it is forced that $\omega = ( \, 1,2, \ldots, m \,)\,$. The desired inclusion follows from observing that for every $k+1 \leq j \leq n-m+1\,$, the value of $\sigma_{j}$ must be $1$ or else the major index of $\sigma$ would be greater than $k\,$.

The reverse inclusion follows by the definitions of the $A_i$'s and $\mathcal{M}_n^{\:\! m+1}(k)\,$.
\end{proof}

\begin{corollary}\label{cor:maj_XXX}
Let $m,n,k$ be nonnegative integers such that $n-m+1\geq k+1\,$. Also let $A_1,\ldots,A_m$ be as in Lemma \ref{lem:maj_XXX}. Then,
\[ \bigg\vert \, \mathcal{M}_n^{\:\! m+1}(k) 
\, \setminus \, \bigcup_{i \in [m]} A_i \, \bigg\vert \, = \, \left| \, \{ \, \sigma \in \mathcal{M}_{k+1}^{\:\! m+1}(k) \, \mid \, \sigma_{k+1}=1 \, \} \, \right| \, . \] 
\end{corollary}

\begin{proof}
The result follows from observing that for every $\sigma$ in $\mathcal{M}_n^{\:\! m+1}(k) 
\, \setminus \, \cup A_i\,$, the value of elements $\sigma_{k+2},\ldots,\sigma_n$ are fixed and maj($\sigma)\,=\,\text{maj}(\sigma')\,$, where $\sigma'$ is equal to $(\,\sigma_1 \, , \, \ldots \, , \, \sigma_{k+1}\, )\,$.
\end{proof}

\begin{lemma}\label{lem:maj_XXX_cap}
Let $m,n,k$ be nonnegative integers such that $n-m+1\geq k+1\,$. Also let $A_1,\ldots,A_m$ be as in Lemma \ref{lem:maj_XXX}. If $J$ is a subset of $[m]$ with $|J|=i\,$, then
\[ \left| \, \bigcap_{j \in J} A_j \, \right| \, = \, \left| \, \mathcal{M}_{n-i}^{m+1}(k) \, \right| \, .  \]
\end{lemma}

\begin{proof}
Let $\varphi \colon \cap A_j \rightarrow \mathcal{M}_{n-i}^{m+1}(k)$ via $\sigma \mapsto \overline{\sigma}\,$, where $\sigma$ is $( \, \sigma_1, \, \ldots, \, \sigma_{n-m} \, , \, \omega_1 \, , \, \ldots \, , \, \omega_m \,)$ and $\overline{\sigma}$ is the subsequence of $\sigma$ with $\omega_j$ removed for every $j \in J\,$. Note that $\overline{\sigma}$ is of the proper length for the expressed codomain of $\varphi\,$. Also note that the elements of $\sigma$ whose indices are accounted for by maj($\sigma$) are unaffected by $\varphi\,$: when $|J|< m\,$, the values of $n-m$ is at least $k\,$; when $|J|=m\,$, every $\omega_i$ equals $m+1\,$. As such, the values of maj($\sigma$) and maj($\overline{\sigma}$) are equal. Hence, the image of $\varphi$ is contained within the desired codomain.

To show surjectivity, observe that each $A_j$ in $\cap A_j$ induces a loss of one degree of freedom in the expression of any $\sigma$ from $\mathcal{M}_n^{m+1}(j)\,$. Viewing this loss as being induced on the element $\omega_j\,$, the map $\varphi$ results in $\overline{\sigma}$ being free from the adjacent element equality that is forced by the $A_j$'s.

To show injectivity, consider $\sigma^1,\sigma^2$ in $\cap A_j$ such that $\sigma^1$ and $\sigma^2$ are unequal. Let $a$ be the largest index of element such that $\sigma_a^1$ differs from $\sigma_a^2\,$. If $a$ is greater than $n-m\,$, the result follows from observing that $\sigma_a^1$ and $\sigma_a^2$ are necessarily not among the $\omega_j$ removed by $\varphi\,$. Should $a$ be at most $n-m\,$, the result follows given that such $\sigma_a^1$ and $\sigma_a^2$ are unaffected by $\varphi\,$.
\end{proof}
$\;$\\[-1.5em]

\subsection{The Insertion Method}

We will now describe a construction, called The Insertion Method, that provides a bridge between fundamental sequences and the major index statistic. This construction, first developed by Carlitz \cite{Carlitz} and later clarified by Wilson \cite{Wilson}, will be an essential component of the proof for the upcoming theorem.

Let $m,n$ be nonnegative integers, and let $(F_1, \ldots, F_m)$ be a fundamental sequence in $\mathcal{F}_n^{\:\! m}\,$. For every $v$ in $[m]\,$, list the elements of $F_v$ in nonincreasing order, labeling them as $f_{v,1} \geq \cdots \geq f_{v,k_v}$ where $k_v$ equals $|F_v|\,$. The sequence $(\, f_{1,1} \, , f_{1,2} \, , \ldots, f_{m,k_m} \, )$ will be referred to as $\tau = ( \, \tau_1 ,\ldots, \tau_n \, )\,$. Also define the value function $v \colon [n] \rightarrow [m]$ such that $v(i)$ equals $j\,$, where $\tau_i$ corresponds to its respective $f_{j,k}\,$. We will build a sequence $\sigma$ in $\mathcal{S}_n^{\:\! m}$ inductively using $\tau$ and $v\,$.

Let $\sigma^1 \, = \, \left( \, v(1) \, \right)\,$. For every $2\leq i \leq n\,$, there is some $a \in [i]$ such that $\sigma_a^i$ equals $v(i)\,$. Moreover, the sequence $\sigma^i$ shall be of the form \[ \sigma_b^i \, = \, \begin{cases} v(i) \;\; \text{when} \;\; b=a \: , \\[2pt] \sigma_b^{i-1} \;\; \text{when} \;\; 1\leq b < a \: , \\[2pt] \sigma_{b-1}^{i-1} \;\; \text{when} \;\; a < b \leq i \: .  \end{cases} \] The value $a$ shall be determined by the following process: \begin{enumerate}
    \item Label $\sigma^i_i$ with a zero. \\[-1.6em]
    \item Working greatest to least among $j$ in $[i-2]\,$, for every $\sigma_j^{i-1} > \sigma_{j+1}^{i-1}\,$ label $\sigma_{j+1}^i$ with successively increasing positive integers $1,2,3,\ldots, d$. \\[-1.6em]
    \item Working least to greatest among $j$ in $[i-1]\,$, if $\sigma_j^i$  is currently unlabeled, label $\sigma_j^i$ with successively positive integers $d+1, d+2, \ldots, i-1\,$. \\[-1.6em]
    \item Find the $\sigma_j^i$ labeled with a $\tau_i\,$, and let $a$ equal $j\,$.
\end{enumerate}

\begin{example}\label{ex:inductive_maj}
Consider the fundamental sequence \[ \big( \, F_1,F_2,F_3,F_4 \, \big) \, = \, \big( \: \{ 0,0 \} \, , \, \{ 1 \} \, , \, \{2,3\} \, , \{1,5\} \: \big) \, . \] Note the contents of Figure \ref{fig:inductive_ex}.

\begin{figure}[!htb]
\begin{center}
\begin{tabular}{c|c|c|c|c|c}
   $i$ & $\tau_i$ & $v(i)$ & Labeling for $\sigma^i$ & $\sigma^i$ & maj($\sigma^i$)  \\[2pt] \hline\hline
    1 & 0 & 1 &  & (\,1\,) & 0 \\[1pt] \hline
    2 & 0 & 1 & (\,{\small 1\,,\,0}\,) & (\,1\,,\,1\,) & 0 \\[1pt] \hline
    3 & 1 & 2 & (\,{\small 1\,,\,2\,,\,0}\,) & (\,2\,,\,1\,,\,1\,) & 1 \\[1pt] \hline
    4 & 3 & 3 & (\,{\small 2\,,\,1\,,\,3\,,\,0}\,) & (\,2\,,\,1\,,\,3\,,\,1\,) & 4 \\[1pt] \hline
    5 & 2 & 3 & (\,{\small 3\,,\,2\,,\,4\,,\,1\,,\,0}\,) & (\,2\,,\,3\,,\,1\,,\,3\,,\,1\,) & 6 \\[1pt] \hline
    6 & 5 & 4 & (\,{\small 3\,,\,4\,,\,2\,,\,5\,,\,1\,,\,0}\,) & (\,2\,,\,3\,,\,1\,,\,4\,,\,3\,,\,1\,) & 11 \\[1pt] \hline
    7 & 1 & 4 & (\,{\small 4\,,\,5\,,\,3\,,\,6\,,\,2\,,\,1\,,\,0}\,) & (\,2\,,\,3\,,\,1\,,\,4\,,\,3\,,\,4\,,\,1\,) & 12 \\[1pt] \hline
\end{tabular}
\end{center}
\caption{\label{fig:inductive_ex} The construction of $\sigma$ for Example \ref{ex:inductive_maj}}
\end{figure}
\end{example}

Proof of the facts $\text{maj}(\sigma^i) \, = \, \text{maj}(\sigma^{i-1})+\tau_i$ and that The Insertion Method provides a bijection from $\mathcal{F}_n^{\:\! m}$ to $\mathcal{S}_n^{\:\! m}$ are omitted here as they are contained in \cite{Carlitz}.

\begin{proposition}\label{prop:F_n^m(k)}
Let $m,n,k$ be nonnegative integers, and consider the following sets
\begin{align*}
    (i)& \;\; \{ \, (F_1,\ldots,F_m) \in \mathcal{F}_n^{\:\! m}(k) \, \mid \, 0\not\in F_2\, ,\ldots, \, 0 \not\in F_m \, \} \, ; \\
    (ii)& \;\; \{ \, \sigma\in \mathcal{S}_n^{\:\!m} \, \mid \, \emph{maj}(\sigma)=k \;\; \emph{and} \;\; \sigma_n=1 \, \} \, .
\end{align*}
When restricted to $(i)\,$, The Insertion Method provides a bijection onto $(ii)\,$.
\end{proposition}

%\begin{proof}
%We will first show that The Insertion Method maps $(i)$ into $(ii)\,$. Fix a fundamental sequence in $(i)\,$. Note that every $v(i)$ which is greater than $1$ is associated with a $\tau_i$ that is nonzero. Since $\sigma_i^i$ is labeled with a zero, the value of $\sigma_i^i$ is not $v(i)\,$. It follows that $\sigma_n$ must be a $1\,$. Since  $\text{maj}(\sigma^i) \, = \, \text{maj}(\sigma^{i-1})+\tau_i\,$, it follows that $\text{maj}(\sigma)$ is equal to $\tau_1+\cdots+\tau_n\,$, which equals $k\,$.

%Since The Insertion Method is injective, it remains to show surjectivity. Let $\sigma$ be a sequence in $(ii)\,$. Since $\sigma_n$ is equal to $1\,$, it follows that $\sigma_i^i$ is not $v(i)$ when $v(i)$ is $2$ or greater. As such, every $v(i)$ greater than $1$ is associated with a $\tau_i$ that is nonzero.
%\end{proof}

\begin{proof}
Consider the following two biconditional statements:
\begin{center}
$\forall \: v(i)>1\, , \;\, \tau_i \neq 0 \;\;\; \Longleftrightarrow \;\;\; \sigma_i^i = 1\, , \;\, \forall \: i \in [n] \;\;\; \Longleftrightarrow \;\;\; \sigma_n=1 \,$; \\[8pt]
$\text{maj}(\sigma^i) \, = \, \text{maj}(\sigma^{i-1})+\tau_i \;\;\; \Longleftrightarrow \;\;\; \text{maj}(\sigma) \, = \, \tau_1 +\cdots+\tau_n \,$.
\end{center}
It follows that The Insertion Method maps $(i)$ into $(ii)$ and does so bijectively.
\end{proof}

\subsection{Generalized Galois Numbers}

We will now define a generalized Galois number, which can be found in \cite{Vinroot} and whose coefficients will be the central object of our theorem.

\begin{definition}\label{def:galois_num}
If $m,n$ are nonnegative integers, then \[ G_n^{\:\!m} \, \coloneqq \, \sum_{\substack{k_1+ \cdots +k_m \:\! = \:\! n \\[2pt] k_i \geq 0}} {n \choose {k_1,\ldots,k_m}}_q \: . \] This polynomial can be referred to as the \textbf{generalized Galois number of $\bm{(m,n)}$}\,.
\end{definition}

\begin{figure}[!htb]
\begin{align*}
G_1^{\:\! 3} \, &= \, 3 \,  &  G_2^{\:\! 3} \, &= \, 3q+6 \,  \\[12pt]
G_3^{\:\! 3} \, &= \, q^3+8q^2+8q+10\,   &  G_4^{\:\! 3} \, &= \, 3q^5 + \cdots + 18q^3 + 21q^2 + 15q +15 \,  \\[12pt]
G_5^{\:\! 3} \, &= \, 3q^8 + \cdots +45q^3 + 39q^2 + 24q +21\, \hspace{0.1in}  & G_6^{\:\! 3} \, &= \, q^{12} + \cdots + 82q^3 + 62q^2+35q + 28 \, 
\end{align*}
\caption{\label{fig:finite_diff} Generalized Galois numbers $G_1^{\:\! 3}\, , \, \ldots \, , G_6^{\:\! 3}\,$.}
\end{figure}

The generalized Galois numbers of $(2,n)$ are precisely the Galois numbers that were defined by Goldman and Rota \cite{rota}. Figure \ref{fig:finite_diff} contains examples of generalized Galois numbers of $(3,n)$ that were calculated using a recursive relation from \cite{Vinroot}.

\begin{proposition}\label{prop:major_GaloisNumber}
If $m,n$ are nonnegative integers, then
\[ G_n^{\:\! m} \, = \, \sum_{\sigma \in \mathcal{S}_n^{\:\! m}} q^{\emph{inv}(\sigma)} \, = \, \sum_{\sigma \in \mathcal{S}_n^{\:\! m}} q^{\emph{maj}(\sigma)} \, . \]
\end{proposition}

\begin{proof}
The result follows from Definition \ref{def:galois_num} and Equation (\ref{eq:MacMahon}).
\end{proof}
$\;$\\[-1.5em]

One final well-known definition is needed to concisely state the theorem.

\begin{definition}\label{def:nabla}
Let $f \colon \mathbb{Z} \rightarrow \mathbb{Z}$ be a function, and define \textbf{the finite difference of $\bm{f}$} to be \[ \nabla f \colon \mathbb{Z} \rightarrow \mathbb{Z} \quad \emph{via} \quad n \mapsto f(n)-f(n-1) \, . \] \end{definition}

Inductively defining the $m^{\text{th}}$-finite difference of $f$ to be $\nabla^m f \, \coloneqq \, \nabla\left( \, \nabla^{m-1} f \, \right)\,$ for any positive integer $m\geq 2\,$, a well-known result follows
\begin{equation}\label{eq:m^th_difference} \nabla^m \:\! f(n) \, = \, \sum_{i=0}^m \:\! (-1)^{i} \:\! {{m} \choose i} \:\! f(n-i) \, . \end{equation}

Letting $f_k^{\:\! 3}(n)$ be the coefficient of $q^k$ in the simplified polynomial $G_n^{\:\! 3}\,$, Figure \ref{fig:finite_diff_comp} contains some example finite difference computations using the polynomials in Figure \ref{fig:finite_diff}.

\begin{figure}[!htb]
\begin{center}
$\begin{array}{ccccccc}
\nabla^2\:\! f_2^{\:\! 3}(3) \, = \, 8 \,  & \hspace{.25in} & \nabla^2\:\! f_2^{\:\! 3}(4) \, = \, 5 \,  & \hspace{.25in} & \nabla^2\:\! f_2^{\:\! 3}(5) \, = \, 5 \,  & \hspace{.25in} & \nabla^2\:\! f_2^{\:\! 3}(6) \, = \, 5 \,   \\[12pt]
\nabla^2\:\! f_3^{\:\! 3}(3) \, = \, 1 \,  & \hspace{.25in} & \nabla^2\:\! f_3^{\:\! 3}(4) \, = \, 16 \,  & \hspace{.25in} & \nabla^2\:\! f_3^{\:\! 3}(5) \, = \, 10 \,  & \hspace{.25in} & \nabla^2\:\! f_3^{\:\! 3}(6) \, = \, 10 \, 
\end{array}$
\end{center}
\caption{\label{fig:finite_diff_comp} Sample finite difference computations using $f_k^3(n)\,$.}
\end{figure}

Observe that $\nabla^2\:\! f_2^{\:\! 3}(4) \, , \, \nabla^2\:\! f_2^{\:\! 3}(5) \, , \, \nabla^2\:\! f_2^{\:\! 3}(6)$ are equal to the number of integer partitions of $2$ with $2$ kinds (from Figure \ref{fig:part_kinds}).

\begin{theorem}\label{theorem}
Let $m,n,k$ be nonnegative integers such that $n \geq m+k\,$. Then, \[ \nabla^{m} \:\! f_k^{\:\! m+1}(n) \, = \, \left| \, \mathcal{P}_k^{m}\, \right| \, , \] where $f_k^{m+1}(n)$ evaluates to the coefficient of $q^k$ in the simplified polynomial $G_n^{\:\! m+1}\,$. 
\end{theorem}

\begin{proof}
By the definition of $\mathcal{M}_n^{\:\! m+1}(k)$ in Lemma \ref{lem:maj_XXX}, observe that $f_k^{\:\! m+1}(n-i)$ is equal to $\left| \, \mathcal{M}_{n-i}^{\:\! m+1}(k) \, \right|\,$. Applying this observation and Equation \ref{eq:m^th_difference}, we have that \[ \nabla^{m} \:\! f_k^{\:\! m+1}(n) \, = \, \sum_{i=0}^m \:\! (-1)^{i} \:\! {{m} \choose i} \:\! \left| \, \mathcal{M}_{n-i}^{\:\! m+1}(k) \, \right| \, . \] Note that the assumed relation $n \geq m+k$ satisfies the similar assumption of Lemma \ref{lem:maj_XXX} and Lemma \ref{lem:maj_XXX_cap}. Applying these two lemmas and the Principle of Inclusion and Exclusion, the following equality is yielded \[ \sum_{i=0}^m \:\! (-1)^{i} \:\! {{m} \choose i} \:\! \left| \, \mathcal{M}_{n-i}^{\:\! m+1}(k) \, \right| \, = \, \bigg\vert \, \mathcal{M}_n^{\:\! m+1}(k) 
\, \setminus \, \bigcup_{i \in [m]} A_i \, \bigg\vert \, , \] where $A_1 , \ldots, A_m$ are as defined in Lemma \ref{lem:maj_XXX}. Letting $\mathcal{T}$ be $\{ \, \sigma \in \mathcal{M}_{k+1}^{\:\! m+1}(k) \, \mid \, \sigma_{k+1}=1 \, \}$ and applying Corollary \ref{cor:maj_XXX}, it follows that \[ \nabla^{m} \:\! f_k^{\:\! m+1}(n) \, = \, \left| \, \mathcal{T} \, \right| \, . \] Since the rightmost element of every sequence in $\mathcal{T}$ is 1, Proposition \ref{prop:F_n^m(k)} applies to $\mathcal{T}$ and it follows that \[ \nabla^{m} \:\! f_k^{\:\! m+1}(n) \, = \, \left| \,  \{ \, (F_1,\ldots,F_{m+1}) \in \mathcal{F}_{k+1}^{\:\! m+1}(k) \, \mid  \, 0\not\in F_2\, ,\ldots, \, 0 \not\in F_{m+1} \, \} \, \right| \, . \] In the case that $m$ is greater than $0\,$, further applying Proposition \ref{prop:kinds_fundamental} achieves the desired result. Should $m$ equal $0\,$, the desired result occurs from observing that both sides of the equality in the theorem statement equal zero when $k\neq 0$ and equal 1 when $k=0\,$.
\end{proof}
$\;$\\[-1.5em]

Simply stated, Theorem \ref{theorem} expresses that as $n$ grows the $m^{\text{th}}$ finite difference of $f_k^{\:\! m+1}(n)$ is eventually constant, and the resulting constant is precisely the number of integer partitions of $k$ with $m$ kinds. Reflecting back to Figure \ref{fig:finite_diff_comp}, we can observe that the sample computations of $\nabla^2\:\!f_k^{\:\! 3}(n)$ become constant when $n$ is at least $k+2$ in value.

\begin{corollary}\label{cor:theorem}
If $n,k$ are nonnegative integers such that $n \geq k\,$, then \[ \frac{d^k}{dq^k} \left( \, \frac{G_{n+1}^2 - G_n^2}{k!} \, \right) \bigg\vert_{q=0} \, = \, \emph{part}(k) \, , \] where $\frac{d}{dq}$ is the conventional derivative operator and \emph{part(}$k$\emph{)} is the number of integer partitions.
\end{corollary}

\begin{proof}
Follows directly from Theorem \ref{theorem} and Taylor's Theorem.
\end{proof}

\section{Concluding Remarks}

This research culminated in Theorem \ref{theorem}, which may be interesting to prove using the inversion statistic rather than the major index statistic. Also, it may be insightful to use analytical methods relating to the calculus of finite differences to achieve Theorem \ref{theorem}. Moreover, it may be worth investigating information buried within expressions of $q$-multinomial coefficients different from the Galois numbers. Furthermore, $q$-multinomial coefficients can be generalized to $p,q$-binomial coefficients \cite{p-q}, and many of this paper's results may be able to be extended.

\section*{Acknowledgments}

Special thanks to Alex Foster and Jonathan Winter for generously offering their computer programming skills in assistance of this project, and an additional mention of thanks to Alex for his pointing us to the Online Encyclopedia of Integer Sequences \cite{OEIS}.

%%%%%%%% AUTHORS' INFORMATION. DELETE/ADD AUTHORS AS NEEDED
{\footnotesize  
\medskip
\medskip
\vspace*{1mm} 
 
\noindent {\it Adrian Avalos}\\  
Coastal Carolina University \\
100 Chanticleer Drive East\\
Conway, SC 29528\\
E-mail: {\tt alavalos@coastal.edu}\\ \\  

\noindent {\it Mark Bly}\\  
Coastal Carolina University \\
100 Chanticleer Drive East\\
Conway, SC 29528\\
E-mail: {\tt mbly@coastal.edu}\\ \\

}

% Please do not change the lines below, they will be changed in copyediting
\vspace*{1mm}\noindent\footnotesize{\date{ {\bf Received}: April 31, 2017\;\;\;{\bf Accepted}: June 31, 2017}}\\
\vspace*{1mm}\noindent\footnotesize{\date{  {\bf Communicated by Some Editor}}}

\end{document}